\documentclass[12pt]{amsart}
\usepackage{esint}
\usepackage{amstext}
\usepackage{amsthm}
\usepackage{amsmath}
\usepackage{amssymb}
\usepackage{mathrsfs}
\usepackage{latexsym}
\usepackage{amsfonts}
\usepackage{graphicx}
\usepackage{comment}
\usepackage{bbm}
\usepackage[mathscr]{eucal}

\usepackage{enumerate,color}
\usepackage[pagebackref,hypertexnames=false, colorlinks, citecolor=red, linkcolor=red]{hyperref} %,hypertexnames=false,colorlinks,[pagebackref]
\usepackage{mathtools}
\mathtoolsset{showonlyrefs,showmanualtags}

\def\XXint#1#2#3{{\setbox0=\hbox{$#1{#2#3}{\int}$}
\vcenter{\hbox{$#2#3$}}\kern-.5\wd0}}

\input txdtools

\newcommand{\unit}{\mathbbm 1}
\newcommand{\norm}[1]{\ensuremath{\left\|#1\right\|}}
\newcommand{\ip}[2]{\ensuremath{\left\langle#1,#2\right\rangle}}
\newcommand{\abs}[1]{\ensuremath{\left\vert#1\right\vert}}

\newcommand{\pr}[1]{\ensuremath{\left(#1\right)}}

\newcommand{\MC}[1]{\ensuremath{\mathcal{#1}}}

\newcommand{\D}{\mathcal{D}}

\newcommand{\R}{\mathbb{R}}
\newcommand{\Z}{\mathbb{Z}}
\newcommand{\C}{\mathbb{C}}
\newcommand{\N}{\mathbb{N}}
\newcommand{\F}{\ensuremath{\mathscr{F}}}

\newcommand{\U}{\mathscr{U}}
\newcommand{\V}{\mathscr{V}}

\newcommand{\Cn}{\ensuremath{\mathbb{C}^n}}

\newcommand{\J}{\ensuremath{\mathscr{J}}}

\newcommand{\f}[2]{\frac{#1}{#2}}

\newcommand{\Rd}{\mathbb{R}^d}

\renewcommand{\S}{\ensuremath{\text{Sig}_d}}
\renewcommand{\v}[1]{\ensuremath{\vec{#1}}}
\newcommand{\br}[1]{{\ensuremath{\left[ #1\right]}}}
\newcommand{\BMOVU}{\ensuremath{{{\text{BMO}}_{V, U} ^p}}}
\newcommand{\BMOVUT}{\ensuremath{{{\widetilde{\text{BMO}}}_{V, U} ^p}}}
\newcommand{\BMOvuT}{\ensuremath{{{\widetilde{\text{BMO}}}_{v, u} ^p}}}
\newcommand{\BMOUUT}{\ensuremath{{{\widetilde{\text{BMO}}}_{U, U} ^p}}}
\newcommand{\BMOVUTd}{\ensuremath{{{\widetilde{\text{BMO}}}_{U', V'} ^{p'}}}}

\newcommand{\MOVUT}[1]{\ensuremath{{{\widetilde{\text{MO}}}_{V, U} ^p}(#1)}}

\newcommand{\Mn}{\ensuremath{\mathbb{M}_{n \times n}}}

\newcommand{\inrd}{\ensuremath{\int_{\mathbb{R}^d}}}

\newcommand{\Ap}[1]{\ensuremath{\br{#1}_{\text{A}_p}}}
\newcommand{\AP}[2]{\ensuremath{\br{#1}_{\text{A}_p (#2)}}}
\newcommand{\Ainf}[2]{\ensuremath{\br{#1}_{\text{A}_{#2, \infty}^\text{sc}}}}

\newcounter{vremennyj}

\numberwithin{equation}{section}

\newtheorem{thm}{Theorem}[section]
\newtheorem{lm}[thm]{Lemma}
\newtheorem{cor}[thm]{Corollary}

\newtheorem{prop}[thm]{Proposition}
\newtheorem*{prop*}{Proposition}

\theoremstyle{remark}

\newtheorem*{rem*}{Remark}

\makeatletter
\newcommand*{\defeq}{\mathrel{\rlap{%
                     \raisebox{0.3ex}{$\m@th\cdot$}}%
                     \raisebox{-0.3ex}{$\m@th\cdot$}}%
                     =}
										
\newcommand*{\eqdef}{=
										 \mathrel{\rlap{%
                     \raisebox{0.3ex}{$\m@th\cdot$}}%
                     \raisebox{-0.3ex}{$\m@th\cdot$}}%
										}
\makeatother

\begin{document}

\title[Commutators in the two weighted setting]{Commutators in the two scalar and \\ matrix weighted setting}

\author{Joshua Isralowitz}
\address{Department of Mathematics and Statistics\\
University at Albany, 1400 Washington Ave., Albany, NY, 12222.}\email{jisralowitz@albany.edu}

\author{Sandra Pott}
\address{Department of Mathematics\\
Lund University, P.O. Box 118
S-221 00 Lund, Sweden.}\email{sandra@maths.lth.se}

% \thanks{Supported  in part by the National Science Foundation under the grant DMS-1301579.}

\author{Sergei Treil}
\address{Department of Mathematics, \\
Brown University
151 Thayer Street,  Box 1917
Providence, RI 02912}\email{treil@math.brown.edu}

%\subjclass[2010]{Primary 42B20, 60G42, 60G46}

%\keywords{}

\begin{abstract}
In this paper we approach the two weighted boundedness of commutators via matrix weights.  This approach provides both a sufficient and a necessary condition for the two weighted boundedness of commutators with an arbitrary linear operator in terms of one matrix weighted norm inequalities for this operator.  Furthermore, using this approach, we surprisingly provide conditions that almost characterize the two matrix weighted boundedness of commutators with CZOs and completely arbitrary matrix weights, which is even new in the fully scalar one weighted setting.   Finally, our method allows us to extend the two weighted Holmes/Lacey/Wick results to the fully matrix setting (two matrix weights and a matrix symbol), completing a line of research initiated by the first two authors.
\end{abstract}

\maketitle

\hypersetup{linktocpage}
  \setcounter{tocdepth}{1}
\tableofcontents

\section{Introduction and main results}

Let $w$ be a weight on $\Rd$ and let $L^p(w)$ be the standard weighted Lebesgue space with respect to the norm \begin{equation*} \|f\|_{L^p( w)} = \left(\inrd |f(x)|^p w(x) \, dx \right)^\frac{1}{p}. \end{equation*} Furthermore, let A${}_p$ be the  Muckenhoupt class of weights $w$ satisfying \begin{equation*} \sup_{\substack{Q \subseteq \Rd \\ Q \text{ is a cube}}} \left( \fint_Q w (x) \, dx \right)\left( \fint_Q w ^{-\frac{1}{p - 1}} (x) \, dx\right)^{p-1} < \infty \end{equation*} where $\fint_Q$ is the unweighted average over $Q$ (which will also occasionally be denoted by $m_Q$).

Given a weight $\nu$, we say $b \in \text{BMO}_{\nu}$ if  \begin{equation*} \|b\|_{\text{BMO}_\nu} = \sup_{\substack{Q \subseteq \Rd \\ Q \text{ is a cube}} } \frac{1}{\nu(Q)} \int_Q |b(x) - m_Q b | \, dx  < \infty \end{equation*} (where $\nu(Q) = \int_Q \nu$) so that clearly
 $\text{BMO} = \text{BMO}_\nu$ when $\nu \equiv 1$.  Further, given a linear operator $T,$ define the commutator $[M_b, T] = M_b T - TM_b$ with $M_b$ being multiplication by $b$.  In the papers \cite{HLW1,HLW2} the authors extended earlier work of S. Bloom \cite{B} and proved that if $u, v \in \text{A}_p$ and $T$ is any Calder\'{o}n-Zygmund operator (CZO) then \begin{equation} \|[M_b, T]\|_{L^p(u) \rightarrow L^p(v)} \lesssim \|b\|_{\text{BMO}_\nu} \label{HWUpper}  \end{equation} where $\nu = (uv^{-1})^\frac{1}{p}$ and it was proved in \cite{HLW2} that if $R_s$ is the $s^\text{th}$ Riesz transform then \begin{equation} \|b\|_{\text{BMO}_\nu} \lesssim \max_{1 \leq s \leq d } \|[ M_b, R_s]\|_{L^p(u) \rightarrow L^p(v )}.  \label{HWLower}  \end{equation}

The purpose of this paper is to give largely self contained proofs of  \eqref{HWUpper} and \eqref{HWLower} and to extend both to the case of two matrix A${}_p$ weights and a matrix symbol $B$ by using arguments inspired by the matrix weighted techniques developed in \cite{GPTV}. Furthermore, as  byproducts of some of our results, we will  provide both a sufficient and a necessary condition for the two weight boundedness of commutators with an arbitrary linear operator in terms of matrix weighted norm inequalities for this operator. Furthermore, we will provide conditions that almost characterize the two matrix weighted boundedness of commutators with CZOs and completely arbitrary matrix weights, which is even new in the fully scalar one weighted setting.

In particular, let $W : \Rd \rightarrow \Mn$ be  an $n \times n$ matrix weight  (a positive definite a.e. $\Mn$ valued function on $\Rd$) and let $L^p(W)$ be the space of $\Cn$ valued functions $\v{f}$ such that \begin{equation*} \|\v{f}\|_{L^p(W)}  = \left(\inrd |W^\frac{1}{p}(x) \v{f}(x)|^p \, dx \right)^\frac{1}{p} < \infty. \end{equation*}   Furthermore, we will say that a matrix weight $W$ is a matrix A${}_p$ weight (see \cite{R}) if it satisfies \begin{equation} \label{MatrixApDef} \Ap{W} = \sup_{\substack{Q \subset \R^d \\ Q \text{ is a cube}}} \fint_Q \left( \fint_Q \|W^{\frac{1}{p}} (x) W^{- \frac{1}{p}} (y) \|^{p'} \, dy \right)^\frac{p}{p'} \, dx  < \infty. \end{equation}

 Before we state our results, let us rewrite Bloom's BMO condition in a way that naturally extends to the matrix weighted setting.  First, by multiple uses of the A${}_p$ property and H\"{o}lder's inequality, it is easy to see that \begin{align*} m_Q \nu   \approx (m_Q u) ^\frac{1}{p} (m_Q v^{-\frac{p'}{p}})^\frac{1}{p'}   \approx (m_Q u) ^\frac{1}{p} (m_Q v)^{-\frac{1}{p}}  \approx (m_Q u^\frac{1}{p}) (m_Q v^\frac{1}{p})^{-1}  \end{align*} (where again, $m_Q$ denotes unweighted average) so that $b \in \text{BMO}_\nu$ when $u$ and $v$ are A${}_p$ weights if and only if

\begin{equation*} \sup_{\substack{Q \subseteq \R \\ Q \text{ is a cube}} }  \fint_Q (m_Q v^\frac{1}{p}) (m_Q u^\frac{1}{p})^{-1}  |b(x) - m_Q b | \, dx  < \infty. \end{equation*}  Now if $U, V$ are matrix A${}_p$ weights, then we define $\BMOVU$ to be the space of $n \times n$ locally integrable matrix functions $B$  where \begin{equation*} \|B\|_{\BMOVU} = \sup_{\substack{Q \subseteq \Rd \\ Q \text{ is a cube}} }  \pr{\fint_Q \|(m_Q V^\frac{1}{p}) (B(x) - m_Q B)(m_Q U^\frac{1}{p})^{-1} \| \, dx}^\f{1}{p}  < \infty \end{equation*} so that $\|b\|_{\BMOVU}  \approx \|b\|_{\text{BMO}{\nu}}$  if $U, V$ are scalar weights and $b$ is a scalar function. Note that the $\BMOVU$ condition is much more naturally defined in terms of reducing matrices, which will be discussed in the next section.

  In this paper we will prove the following two theorems, the first of which is a generalization of a similar but much weaker result proved in \cite{I}.

\begin{thm} \label{BloomUB}  Let $T$ be any linear operator defined on scalar valued function where its canonical vector-valued extension $T \otimes \mathbf{I}_{n}$ is bounded on $L^p(W)$ for all $n \times n$ matrix A${}_p$ weights $W$ and all $n \in \N$ with bound depending on $T, n, d, p$, and $\Ap{W}$ (which is known to be true for all CZOs, see \cite{CIM} for a very easy proof).  If $U, V$ are $m \times m$ matrix A${}_p$ weights and $B$ is an $m \times m$ locally integrable matrix function for some $m \in \N$,  then \begin{equation*} \|[M_B, T\otimes \mathbf{I}_m]\|_{L^p(U) \rightarrow L^p(V)} \lesssim \|B\|_{\BMOVU} \end{equation*} with bounds depending on $T, m, d, p, \Ap{U}$ and $\Ap{V}$. \end{thm}

\noindent In particular, in the case when $u, v,$ and $b$ are scalar valued (that is, $m = 1$), we have that \eqref{HWUpper} holds for any linear operator $T$  such that $T \otimes \mathbf{I}_n$  is bounded on $L^p(W)$ for all  $n \times n$ matrix A${}_p$ weights $W$ and all $n \in \N$ (and in particular we have \eqref{HWUpper} for all CZOs).

We will need one more definition before we state our second main result.  Let $\mathbf{I}_n$ denote the $n \times n$ identity matrix.  Given a finite collection $R = \{R_s\}_{s = 1}^N$ of linear operators defined on scalar valued functions, we say that $R$ is a lower bound collection if for any $n \in \N$ and any $n \times n$  matrix weight $W$ we have  \begin{equation} \Ap{W} ^\frac{1}{p}  \lesssim \max_{1 \leq s \leq N} \|R_s \otimes \mathbf{I}_n\|_{L^p(W) \rightarrow L^p(W)} \label{LBO} \end{equation} \noindent with the bound independent of $W$ (but not necessarily independent of $n$), and each $R_s \otimes \mathbf{I}_n$ is bounded on $L^p(W)$ if $W$ is a matrix A${}_p$ weight. It should be noted that, as one would expect, the Hilbert transform itself and more generally the collection $\{R_\ell\}_{\ell = 1}^d$ of Riesz transforms are lower bound collections (which will be proved in Lemma \ref{AveLem}.)

\begin{thm} \label{BloomLB}  If $R = \{R_s\}_{s = 1}^N$ is a lower bound collection, then for any $m \times m$ matrix A${}_p$ weights $U, V$ and any  $m \times m$ locally integrable matrix symbol $B$ we have \begin{equation*} \|B\|_{\BMOVU} \lesssim \max_{1 \leq s \leq N}  \|[M_B, R_s \otimes \mathbf{I}_{m }]\|_{L^p(U) \rightarrow L^p(V)}. \end{equation*}  \end{thm}

 %In fact, we will prove that it is enough to consider the testing functions \begin{equation*} \v{v}_{Q, \V{e}} (x) = \unit_Q(x) \abs{W^{-\frac{1}{p}} (x) %\v{e}}^{p'-2} W^{-\frac{2}{p}} (x) \v{e}    \end{equation*} over all cubes $Q \subseteq \Rd$ and all $\v{e} \in \Cn$.  Namely, if $\{R_j\}_{j = 1}$ is %the collection of Riesz transforms, then for any $n \times n$ matrix weight $W$ we have

%\begin{equation} \Ap{W} ^\frac{1}{p} \lesssim \sup_{Q, \v{e}} \pr{\max_{1 \leq s \leq d} \frac{ \norm{(R_s \otimes \textbf{I}_n) \v{v}_{Q, %\v{e}}}_{L^p(W)}}{\norm{\v{v}_{Q, \v{e}}}_{L^p(W)}}} \label{RieszTestCond} \end{equation} which in the case $p = 2$  reduces to the much more natural %looking testing condition \begin{equation*} \Ap{W} ^\frac{1}{p} \lesssim  \sup_{Q, \v{e}} \pr{\max_{1 \leq s \leq d}\frac{ \norm{(R_s \otimes %\textbf{I}_n) \unit_Q W^{-1}\v{e} }_{L^2(W)}}{\norm{\unit_Q W^{-1}\v{e}}_{L^2(W)}}}  \end{equation*}

Let us briefly outline the strategy for proving Theorems \ref{BloomUB} and \ref{BloomLB}.  In the next section, we will use matrix weighted arguments inspired by \cite{GPTV} to prove Theorems \ref{BloomUB} and \ref{BloomLB} in terms of a weighted BMO quantity $\|B\|_{\BMOVUT}$ that is equivalent to $\|B\|_{\BMOVU}$ when $U$ and $V$ are matrix A${}_p$ weights (see Corollary \ref{StrongJNCont}) but is much more natural for more arbitrary matrix weights $U$ and $V$.  More precisely, define  \begin{equation} \|B\|_{\BMOVUT}^p  = \sup_{\substack{Q \subseteq \Rd \\ Q \text{ is a cube}} } \fint_Q \pr{\fint_Q \norm{V^\frac{1}{p} (x) (B(x) - B(y)) U^{-\frac{1}{p}}(y) }^{p'} \, dy }^\frac{p}{p'} \, dx. \label{BMOT} \end{equation}
 \noindent In particular, in the  case of two scalar weights $u, v$ and a scalar symbol $b$,  note that \begin{equation*} \|b\|_{\widetilde{BMO}_{u, v} ^{p }} ^p = \sup_{\substack{Q \subseteq \Rd \\ Q \text{ is a cube}} } \fint_Q \pr{\fint_Q  |b(x) - b(y)|^{p'} u^{-\frac{p'}{p}}(y)  \, dy }^\frac{p}{p'} \, v(x) dx \end{equation*} which has a particulary simple and appealing appearance when $ p = 2$, namely \begin{equation*} \|b\|_{\widetilde{BMO}_{u, v} ^{2}} ^2 = \sup_{\substack{Q \subseteq \Rd \\ Q \text{ is a cube}} } \fint_Q \fint_Q  v(x) |b(x) - b(y)|^{2} u^{-1}(y)  \, dy  dx \end{equation*}

 We will then give relatively short proofs of the following two results in Section \ref{UBSection}.

\begin{lm} \label{IntUB} Let $T$ be any linear operator defined on scalar valued functions where its canonical vector-valued extension $T \otimes \mathbf{I}_{n}$ satisfies \begin{equation*} \|T \otimes \mathbf{I}_{n}\|_{L^p(W) \rightarrow L^p(W)} \leq  \phi (\Ap{W}) \end{equation*} for some positive increasing function $\phi$ (possibly depending on $T, d, n$ and $p$.)   If $U, V$ are $m \times m$ matrix A${}_p$ weights and $B$ is a locally integrable $m \times m$  matrix valued function for some $m \in \N$, then   \begin{equation*} \|[M_B, T \otimes \mathbf{I}_m]\|_{L^p(U) \rightarrow L^p(V)} \leq   \|B\|_{\BMOVUT} \phi\pr{3^{\frac{p}{p' }} \pr{\Ap{U} + \Ap{V} } + 1}  \end{equation*}\end{lm}

\begin{lm} \label{IntLB}  If $R = \{R_s\}_{s = 1}^N$ is a lower bound collection of operators, then for any $m \times m$ matrix A${}_p$ weights $U, V$ and an $m \times m$ matrix symbol $B$ we have \begin{equation*} \|B\|_{\BMOVUT}  \lesssim \max_{1 \leq s \leq N}  \|[M_B, R_s \otimes \mathbf{I}_{n} ]\|_{L^p(U) \rightarrow L^p(V)} \end{equation*}  where the bound depends possibly on $n, p, d$ and $R$ but is independent of $U$ and $V$.  \end{lm}

%While the methods used to prove Lemma \ref{IntUB} can be modified to provide  results for iterated commutators $[M_B, T\otimes\textbf{I}_m]_k := [M_B, %[M_B, T\otimes \textbf{I}_m]_{k-1}]$ (where as usual $[M_B, T \otimes \textbf{I}_m]_0 := T \otimes \textbf{I}_m$), these results are in general very messy %and we are unable to obtain (even qualitative) two weight Bloom type results that generalize those of \cite{LOR}.  Despite this,

Recall that a scalar weight $w$ on $\Rd$  is said to satisfy the A${}_\infty$ condition if we have $$[w]_{\text{A}_\infty} = \sup_{\substack{Q \subseteq \Rd \\ Q \text{ is a cube } }} \frac{1}{w(Q)} \int_Q  M( w\unit_Q )  < \infty $$ where $M$ is the ordinary Hardy-Littlewood maximal function on $\Rd$.   Further, for a matrix weight $U$ we define the ``scalar A${}_\infty$ characteristic" as in \cite{CIM,NPTV} by  \begin{equation*} \Ainf{U}{p} = \sup_{\v{e} \in \Cn} \br{ \abs{U^\frac{1}{p} \v{e}}^p }_{\text{A}_\infty} \end{equation*} which for any $1 < p < \infty$ obviously reduces to the ordinary A${}_{\infty}$ characteristic in the scalar setting.

At the end of Section \ref{UBSection} we will estimate $\|b\|_{\BMOUUT}$ for a scalar function $b$ and a matrix A${}_p$ weight $U$  to  give us the following quantitative version of Theorem \ref{BloomUB}.

\begin{prop} \label{BloomQuant} Assume $T$ satisfies the hypothesis of Lemma \ref{IntUB}.  Then there exists $C$ independent of $U, V, b$ and $T$ where
\begin{equation} \|[M_b, T \otimes \mathbf{I}_m] \|_{L^p(U) \rightarrow L^p(U)} \leq     \|b\|_{\text{BMO}}  \pr{\Ainf{U}{p} + \Ainf{U^{-\frac{p'}{p}}}{p'}  } \phi\pr{C\Ap{U}}  \label{OneWeightMixedQuant}.\end{equation}
\end{prop}

 It is interesting to remark that Lemma \ref{IntUB} and Proposition \ref{BloomQuant}  would provide new quantitative one and two weight commutator bounds in the scalar setting if the ``matrix A${}_p$ conjecture" $$ \|T\otimes \mathbf{I}_m\|_{L^p(W) \rightarrow L^p(W)} \lesssim \Ap{W}^{\max\{1, \frac{1}{p-1}\}}$$ were to hold for all CZOs $T$, even in the case $p = 2$. Also, we will prove that the collection of Riesz transforms form a lower bound operator in Section \ref{WienerSection} by utilizing the Schur multiplier/Wiener algebra ideas from \cite{LT}, and thus recovering \eqref{HWLower}.  In fact, we will show much more and prove the following surprising result.

\begin{thm} \label{RieszThm} Let $\{R_\ell\}_{\ell = 1}^d$ be the collection of Riesz transforms, and let $U$ and $V$ be any (not necessarily A${}_p$) matrix weights.  If $B$ is any locally integrable $m \times m$  matrix valued function then \begin{equation} \max\left\{\|B\|_{\BMOVUT}, \|B\|_{\BMOVUTd}\right\}\lesssim \max_{1 \leq \ell\leq d}  \|[M_B, R_\ell \otimes \mathbf{I}_{m }]\|_{L^p(U) \rightarrow L^p(V)}. \label{RieszThmA}\end{equation}
\end{thm}

\noindent  Moreover, we will show that an Orlicz ``bumped" version of these conditions are sufficient for the general two matrix weighted boundedness of a CZO. In particular, we will prove the following result in Section \ref{UBSection}, which is similar in statement and proof to Lemma $4$ in \cite{IPR}.

\begin{prop} \label{OrliczProp} Let $T$ be a CZO, $U$ and $V$ be any $m \times m$ matrix weights, and $B$ be any locally integrable $m \times m$  matrix valued function.  Let $C$ and $D$ be Young functions with $\overline{C} \in B_{p'}$ and $\overline{D}  \in B_p$ where $\overline{C}$ and $\overline{D}$ are the conjugate Young functions to $C$ and $D$, respectively.  Then \begin{equation*} \|[M_B, T\otimes \mathbf{I}_m]\|_{L^p(U) \rightarrow L^p(V)} \lesssim \min\{\kappa_1, \kappa_2\} \end{equation*}  where
\[
\begin{split}\kappa_{1} & = \sup_{Q}\|\|V^\frac{1}{p} (x) (B(x)-B(y)) U^{-\frac{1}{p}} (y) \|_{C_{x},Q}\|_{D_{y},Q}\\
\kappa_{2} & =\sup_{Q}\|\|V^\frac{1}{p} (x) (B(x)-B(y)) U^{-\frac{1}{p}} (y) \|_{D_{y},Q}\|_{C_{x},Q}
\end{split}
\] \end{prop}

\noindent We refer the reader to Section $5.2$ in \cite{IPR} for the standard Orlicz space related definitions used in the statement of the Proposition \ref{OrliczProp}.

 It is important to emphasize that Theorem \ref{RieszThm} and Proposition \ref{OrliczProp} are new,  even in the scalar $p = 2$ setting of a single weight. It is also interesting to note that formally ``removing" $b$ from the condition $\|b\|_{\BMOvuT} ^p < \infty$ in the case of two scalar weights $u$ and $v$ reduces to the classical two weight A${}_p$ condition $(u, v) \in \text{A}_p$.  From this perspective, $\|b\|_{\BMOvuT} ^p$ can be thought of as a first order analogy of the ``zero order"  condition $(u, v) \in \text{A}_p$.  In particular, it is well known (see \cite{MW}) that $(u, v) \in \text{A}_p$ is necessary for the two weighted norm boundedness of the Hilbert transform, and that an Orlicz bumped version of $(u, v) \in \text{A}_p$ (in particular either of the equivalent conditions in Proposition \ref{OrliczProp} when again $b$ is ``removed") is sufficient for the two weighted boundedness of any CZO $T$ , see \cite{L}.  Thus, Theorem \ref{RieszThm} and Proposition \ref{OrliczProp} should be thought of as a first order commutator version of the well known ``zero order" scalar results above.

A key tool for the proof of  Proposition \ref{OrliczProp} is a new convex body domination theorem, which is interesting in its own right and therefore stated here. It was
  essentially proven in \cite{IPR} (though not explicitly stated) and is  more suitable for us here than the sparse convex body domination of commutators in Theorem $4$ from \cite{IPR}.

To state the result, we need some notation.
 Let $\D$ be a dyadic grid of cubes in $\Rd$. Recall that $\mathcal{S}\subset\D$
is a sparse family if for every $Q\in\mathcal{S}$ there exists
a measurable subset $E_{Q}\subset Q$ such that
\begin{enumerate}
\item $|Q|\leq 2|E_{Q}|.$
\item The sets $E_{Q}$ are pairwise disjoint.
\end{enumerate}

 \begin{thm} \label{SparseLem} Let $T$ be a CZO.
For every $\C^m$ valued function $\v{f}$ with compact support and every $m \times m$ valued matrix function $B$ such that $B\v{f} \in L^1$, there exists   $3^d$ sparse collections $\MC{S}_j$ of
dyadic cubes, a constant $c_{d, m, T}$, and for each $Q \in \MC{S}_j$ a function $k_Q : Q \times Q \rightarrow \mathbb{R}$ with $\|k_Q\|_{L^\infty (Q \times Q)} \leq 1$  such that
\begin{equation} \label{SparseEq}
  [M_B,T \otimes \mathbf{I}_m]\v{f}(x)
  = c_{d,m, T}\sum_{j=1}^{3^{d}}\sum_{Q\in\mathcal{S}_{j}} \unit_Q (x) \fint_Q k_Q(x, y) (B(x) - B(y)) \v{f}(y) \, dy \quad (x \in \R^d).
\end{equation}

\end{thm}
Note that this result is even new in the scalar case.  It is important to remark that even in the scalar setting, it seems unclear whether the by now standard ideas from the proof of Theorem $1.1$ in
\cite{LOR1} can be used to prove our sparse domination.  A version of our sparse domination for iterated commutators will be the subject of a future paper.

In the last section we will prove the equivalence of the quantities $\|B\|_{\BMOVUT}$ and $\|B\|_{\BMOVU}$ when $U$ and $V$ are matrix A${}_p$ weights, completing a line of work initiated in \cite{I,IKP}. Additionally we will prove that the quantities $\|B\|_{\BMOVUT}$ and $\|B\|_{\BMOVUTd}$ are equivalent again when $U$ and $V$ are matrix A${}_p$ weights. In particular we will make use of the ideas and results from \cite{I,IKP} in conjunction with an ``extrapolation of inverse H\"{o}lder inequality" argument from \cite{T}.  For the sake of completion, however,  we will reprove all relevant results from \cite{I,IKP}, the proofs of which are more technical than those in Sections \ref{UBSection} and \ref{WienerSection}.

We will end this introduction with three remarks and an outline of the organization of the rest of the paper. First, it is an obvious question as to whether the techniques and results of this paper can be extended to the iterated commutator setting, and whether we can recover the more recent iterated commutator Bloom type results from \cite{LOR2} or the very recent unweighted two symbolled iterated commutator results of \cite{HLO}. This will be pursued in a forthcoming paper.  Second, for the reader who is either unfamiliar with matrix weighted norm inequalities or is primarily interested in the implications of our results in the scalar setting, we have attempted to make this paper almost entirely self contained.

 Third, if \begin{align*} & \lambda_1 = \sup_Q \pr{\fint_Q \norm{V^\frac{1}{p} (x) (B(x) - m_Q B) \MC{U}_Q ^{-1}} ^p \, dx}^\frac{1}{p} \\
&  \lambda_2 = \sup_Q \pr{\fint_Q \norm{U^{-\frac{1}{p}} (x) (B^*(x) - m_Q B^*) (\MC{V}_Q ') ^{-1}} ^{p'} \, dx}^\frac{1}{p'} \end{align*}
   where $\MC{U}_Q$ is an $L^p$ reducing matrix for $U$ on $Q$ and $\MC{V}_Q '$ is an $L^{p'}$ reducing matrix for $V^{-\frac{1}{p}}$ on $Q$ (again, see Section \ref{UBSection}), then an easy use of H\"{o}lder's inequality (see the proof of Corollary \ref{StrongJNCont}) says that $ \lambda_1 \lesssim \|B\|_{\BMOVUT}$ and $\lambda_2 \lesssim \|B\|_{\BMOVUTd}$ for arbitrary matrix weights $U$ and $V$ (and as previously mentioned, all four quantities are equivalent for matrix A${}_p$ weights $U$ and $V$, see Corollary \ref{StrongJNCont}).   Additionally, in the purely scalar two weighted setting, we have that \begin{align*} & \lambda_1 = \sup_Q \pr{\frac{1}{u(Q)}\int_Q \abs{b(x) - m_Q b}^p v(x) \, dx}^\frac{1}{p} \\ &
 \lambda_2 = \sup_Q \pr{\frac{1}{v^{-\frac{p'}{p}}(Q)} \int_Q \abs{ b(x) - m_Q b } ^{p'} u^{-\frac{p'}{p}}(x) \, dx}^\frac{1}{p'} \end{align*}  which proves very natural arbitrary two scalar weighted necessary conditions for the boundedness of commutators with all of the Riesz transforms.

   Also, we can prove a version of Proposition \ref{OrliczProp} involving subtracted averages.  Namely,  arguing in a manner very similar to the proof of Lemma  $4$ of \cite{IPR} and the proof of Proposition \ref{OrliczProp} we have that if $C, D, E, F$ are Young function with $\bar{C},  \bar{E} \in B_{p'}$ and $\bar{D}, \bar{F} \in B_{p}$, then
\begin{equation*}
\|[M_B, T\otimes \mathbf{I}_m]\|_{L^p(U) \rightarrow L^p(V)} { \lesssim} \Lambda_{1}+\Lambda_{2}
\end{equation*}

where $\Lambda_{1}=\min\left\{ \mu_{1},\mu_{2}\right\} $
with
\[
\begin{split}\mu_{1} & =\sup_{Q}\|\|V^{\frac{1}{p}}(x)(B(x) - m_Q B) U^{-\frac{1}{p}}(y)\|_{E_{x},Q}\|_{F_{y},Q}\\
\mu_{2} & =\sup_{Q}\|\|V^{\frac{1}{p}}(x)(B(x) - m_Q B) U^{-\frac{1}{p}}(y)\|_{F_{y},Q}\|_{E_{x},Q}
\end{split}
\]
 and $\Lambda_{2} = \min\left\{ \mu_{3},\mu_{4}\right\} $
with
\[
\begin{split}\mu_{3} & =\sup_{Q} \|\|V^{\frac{1}{p}}(x) (B(y) - m_Q B) U^{-\frac{1}{p}}(y)\|_{C_{x},Q}\|_{D_{y},Q}\\
\mu_{4} & =\sup_{Q}\|\|V^{\frac{1}{p}}(x)(B(y) - m_Q B)U^{-\frac{1}{p}}(y)\|_{D_{y},Q}\|_{C_{x},Q}.
\end{split}
\]

\noindent which in the unbumped (i.e. when $C(x) = E(x) =  x^{p}/p$ and $D(x) = F(x) =x^{p'}/p'$) scalar two weighted setting gives us

 \begin{align*} & \Lambda_{1} \approx \pr{m_Q u^{-\frac{p'}{p}}} ^\frac{1}{p'} \pr{\fint_Q |b(x) - m_Q b|^p v(x)  \, dx }^\frac{1}{p} \\ &
 \Lambda_{2}  \approx \pr {m_Q v}^\frac{1}{p} \pr{\fint_Q |b(x) - m_Q b|^{p'} u^{-\frac{p'}{p}} (x) \, dx }^\frac{1}{p'}
 \end{align*} which are natural joint BMO/A${}_p$ conditions.    Further, by adding and subtracting $m_Q B$, it is trivial that in general $\kappa_1 \lesssim \mu_1 + \mu_2$ and $\kappa_2 \lesssim \mu_3 + \mu_4$ when $C =  E$ and $D = F$. Despite all this, it seems unclear what the precise connection between all of these weighted (umbumped) BMO conditions are when dealing with not necessarily matrix A${}_p$ weights (even in the one weighted fully scalar setting.)

Finally, the paper is organised as follows. In Section \ref{UBSection} we will prove Lemma \ref{IntUB},  Lemma \ref{IntLB}, Proposition \ref{BloomQuant},  Proposition \ref{OrliczProp}, and Theorem \ref{SparseLem}.  In Section \ref{WienerSection} we will prove Theorem \ref{RieszThm}, and in the last section we will prove the equivalence of the quantities $\|B\|_{\BMOVUT}$ and $\|B\|_{\BMOVU}$ when $U$ and $V$ are matrix A${}_p$ weights, which will complete the proofs of Theorem \ref{BloomUB} and Theorem \ref{BloomLB}.
%Moreover, for the reader primarily interested in the scalar setting (and do to the technical nature of our stopping time arguments), at the beginning of %Section \ref{JN} we will give a quick proof of the equivalence of the quantities $\|B\|_{\BMOVUT}$ and $\|B\|_{\BMOVU}$ in the purely scalar setting (when %$U, V$ and $B$ are scalar valued) using the classical weighted John-Nirenberg inequality from \cite{MW} and in conjunction with some results from %\cite{HLW2}.

\section{Intermediate upper and lower bounds} \label{UBSection}

In this section we will give a short proofs of Theorem \ref{SparseLem}, Lemma \ref{IntUB}, Lemma \ref{IntLB}, and Proposition \ref{OrliczProp}, starting with Lemma \ref{IntUB}.

  \subsection{Proof of Lemma \ref{IntUB}}

   Define the $2 \times 2$ block matrix function $\Phi$ by \begin{equation}  \Phi = \left(\begin{array}{cc} V^\frac{1}{p} & V^\frac{1}{p}  B \\ 0 & U^\frac{1}{p} \end{array} \right) \label{Phi} \end{equation}  so that  \begin{equation*} \Phi^{-1} =  \left(\begin{array}{cc} V^{-\frac{1}{p}} & - B U^{-\frac{1}{p}}  \\ 0 & U^{-\frac{1}{p}}. \end{array} \right) \end{equation*} and \begin{equation*} \Phi \pr{T \otimes \mathbf{I}_{2m}}  \Phi^{-1} = \Phi  \left( \begin{array}{cc} T \otimes \mathbf{I}_m & 0 \\ 0 & T \otimes \mathbf{I}_m \end{array}\right) \Phi^{-1} = \left(\begin{array}{cc} V^\frac{1}{p} \pr{T \otimes \mathbf{I}_m} V^{-\frac{1}{p}} & V^\frac{1}{p}  [M_B, T \otimes \mathbf{I}_m] U^{-\frac{1}{p}}   \\ 0 & U^{\frac{1}{p}} \pr{T \otimes \mathbf{I}_m} U^{-\frac{1}{p}} \end{array} \right). \end{equation*}

Let   $W = (\Phi ^* \Phi)^\frac{p}{2} $.  Then using the polar decomposition, we can write \begin{equation*} \Phi =  \MC{U} W^\frac{1}{p} \end{equation*} where $\MC{U}$ is unitary valued a.e.  Supposing that $W$ is a $2m \times 2m$ matrix A${}_p$ weight,  we have by assumption that

\begin{align*} \|[M_B, T \otimes \mathbf{I}_m]\|_{L^p(U) \rightarrow L^p(V)} & = \|V^\frac{1}{p}  [M_B, T \otimes \mathbf{I}_m] U^{-\frac{1}{p}} \|_{L^p \rightarrow L^p}
\\ & \leq \|\Phi \pr{T \otimes \mathbf{I}_{2m}}  \Phi^{-1}\|_{L^p \rightarrow L^p}
\\ & = \|W^\frac{1}{p} \pr{T \otimes \mathbf{I}_{2m}}  W^{-\frac{1}{p}}\|_{L^p \rightarrow L^p}
\\ & = \| T \otimes \mathbf{I}_{2m}  \|_{L^p(W) \rightarrow L^p(W)}
\\ & \leq \phi(\Ap{W}) \end{align*}

To finish the proof of Lemma \ref{IntUB}, note that \begin{align*}
\Phi(x)\Phi(y)^{-1} =
\left( \begin{array}{cc} V^{\frac{1}{p}}(x) V^{-\frac{1}{p}}(y) & V^{\frac{1}{p}}(x) ( B(x) - B(y) ) U^{-\frac{1}{p}} (y) \\ 0 & U^{\f{1}{p}}(x) U^{-\f{1}{p}} (y) \end{array}\right)
\end{align*}

so that \begin{align*}  \fint_Q & \left( \fint_Q \|W^{\frac{1}{p}} (x) W^{- \frac{1}{p}} (y) \|^{p'} \, dy \right)^\frac{p}{p'} \, dx
\\ & =  \fint_Q \left( \fint_Q \|\Phi (x) \Phi^{-1} (y) \|^{p'} \, dy \right)^\frac{p}{p'} \, dx
\\ & \leq 3^\frac{p }{p'} \pr{\Ap{U} + \Ap{V} + \fint_Q \left( \fint_Q \|V^{\frac{1}{p}}(x) ( B(x) - B(y) ) U^{-\frac{1}{p}} (y)
 \|^{p'} \, dy \right)^\frac{p}{p'} \, dx}\end{align*}
 and thus

 \begin{align*} \|[M_B, T \otimes \mathbf{I}_m]\|_{L^p(U) \rightarrow L^p(V)} \leq  \phi\pr{3^{\frac{p}{p' }}\pr{\Ap{U} + \Ap{V} + 1}} \end{align*}

Clearly we may assume that $0 < \|B\|_\BMOVUT < \infty$, so rescalling with $B \mapsto B \|B\|_\BMOVUT  ^{-1}$ completes the proof.

\subsection{Proof of Theorem \ref{SparseLem} and Proposition \ref{OrliczProp}}

\begin{proof}[Proof of Theorem \ref{SparseLem}]

  Define the $\C^{2m}$ valued function $\tilde{f}$ by $$\tilde{f}(x) = \begin{pmatrix} \v{f}(x) \\ \v{f}(x) \end{pmatrix}$$ and define the $2 \times 2$ block matrix $\Phi (x) $ by $$\Phi(x) = \begin{pmatrix} \mathbf{I}_{m } &{  B(x)}    \\ 0 & \mathbf{I}_{m }  \end{pmatrix} $$ so that $$\Phi^{-1} (x) = \begin{pmatrix}\mathbf{I}_{m } & - {  B(x)}   \\ 0 & \mathbf{I}_{m } \end{pmatrix} $$

  Direct computation shows $$\Phi(x) ((T \otimes \mathbf{I}_m) \Phi^{-1} \tilde{f}) (x) = \begin{pmatrix}   (T \otimes \mathbf{I}_m) \v{f}(x) +  [M_B, T \otimes \mathbf{I}_m] \v{f}(x) \\ (T \otimes \mathbf{I}_m)\v{f} (x) \end{pmatrix} $$    and  $$\Phi^{-1} (y) \tilde{f}(y) = \begin{pmatrix}\mathbf{I}_m & - B(y)    \\ 0 & \mathbf{I}_m \end{pmatrix} \begin{pmatrix} \v{f}(y) \\ \v{f}(y) \end{pmatrix} = \begin{pmatrix} \v{f}(y) - B(y) \v{f}(y)   \\ \v{f}(y) \end{pmatrix} $$
  Since $\Phi^{-1} \tilde{f} \in L_c ^1$ , Theorem $3.4$ in \cite{NPTV} applied to $\Phi^{-1} \tilde{f}$ then says that there exists $3^d$ sparse collections $\MC{S}_j$ of dyadic cubes, a constant $c_{d, m, T}$, and for each $Q \in \MC{S}_j$ a function $k_Q : Q \times Q \rightarrow \mathbb{R}$ with $\|k_Q\|_{L^\infty (Q \times Q)} \leq 1$  such that \begin{align*} & \begin{pmatrix} (T \otimes \mathbf{I}_m)\v{f}(x) + [M_B, T \otimes \mathbf{I}_m] \v{f}(x) \\ (T \otimes \mathbf{I}_m)\v{f} (x) \end{pmatrix}  \\ &  =  c_{d, m, T}  \sum_{j = 1}^{3^d} \sum_{Q \in \MC{S}_j} \Phi(x) \begin{pmatrix}  \fint_Q  k_{Q} (x, y) (\v{f}(y) - B(y) \v{f}(y)) \, dy  \\ \fint_Q k_{Q} (x, y)  \v{f}(y) \, dy    \end{pmatrix} \unit_Q(x)
 \\ & = c_{d, m, T} \sum_{j = 1}^{3^d} \sum_{Q \in \MC{S}_j}  \begin{pmatrix}  \fint_Q  k_{Q} (x, y) (\v{f} (y) - B(y) \v{f}(y) + B(x) \v{f}(y)) \, dy \\ \fint_Q k_{Q} (x, y)  \v{f}(y) \, dy    \end{pmatrix} \unit_Q(x). \end{align*}

   Subtracting $$ (T \otimes \mathbf{I}_m)\v{f} (x) = c_{d, n, T} \sum_{j = 1}^{3^d} \sum_{Q \in \MC{S}_j}  \fint_Q k_{Q} (x, y)  \v{f}(y) \, dy$$ from the first column then completes the proof.

\end{proof}

We now prove Proposition \ref{OrliczProp}. The easy proof is similar to the proof of Lemma $4$ in \cite{IPR}.   We only prove that \begin{equation*} \|[M_B, T \otimes \mathbf{I}_m]\|_{L^p(U) \rightarrow L^p(V)} \lesssim \sup_{Q}\|\|V^\frac{1}{p} (x) (B(x)-B(y)) U^{-\frac{1}{p}} (y) \|_{C_{x},Q}\|_{D_{y},Q} \end{equation*} as the other estimate is virtually the same.

By the density of bounded functions with compact support in $L^p(W)$ for any matrix weight $W$ (see Proposition $3.6$ in \cite{CMR}), we can pick $\v{f}, \v{g}$ bounded with compact support and use \eqref{SparseEq} to get that (where for notational ease we supress the summation over $j = 1$ to $3^d$ )

\begin{align*} & \abs{\ip{[M_B, T \otimes \mathbf{I}_m] \v{f}}{\v{g}}_{L^2}}
\\ & \leq  \sum_{Q\in\mathcal{S}}\int_Q \fint_Q \abs{\ip{(B(x) - B(y)) \v{f}(y)}{\v{g}(x)}} \, dy dx  \\ & \leq  \sum_{Q\in\mathcal{S}} \int_Q \fint_Q \norm{V^\frac{1}{p} (x) (B(x) - B(y)) U^{-\frac{1}{p}} (y)} \abs{U^\frac{1}{p} (y) \v{f}(y)} \abs{V^{-\frac{1}{p}} (x) \v{g}(x)} \, dx \,  dy
\\ & \leq 2\pr{\sup_Q \norm{\norm{V^\frac{1}{p} (x) (B(x) - B(y)) U^{-\frac{1}{p}} (y)}_{C_x, Q} }_{D_y, Q}}
\sum_{Q\in\mathcal{S}} |E_Q|  \norm{V^{-\frac{1}{p}}  \v{g}}_{\overline{C}, Q} \norm{U^\frac{1}{p}  \v{f}}_{\overline{D}, Q}  \\ &
\leq 2 \kappa_1  \|M_{\overline{D}} (U^\frac{1}{p}  \v{f})\|_{L^p}  \|M_{\overline{C}} (V^{-\frac{1}{p}}  \v{g})\|_{L^{p'}} \\ &
\lesssim \kappa_1 \|\v{f}\|_{L^p(U)}  \|\v{g}\|_{L^{p'} (V^{-\frac{p'}{p}})} \end{align*} where the last line follows from the fact that $\overline{C} \in \text{B}_{p'}$ and $\overline{D} \in \text{B}_p$. This completes the proof.
  \subsection{Proof of Lemma \ref{IntLB}}

We now give a short proof of Lemma \ref{IntLB}.   Defining $W$ and $\Phi$ as before, we have by the previous computations and by assumption that

\begin{align*}  \bigg(\Ap{U} & + \Ap{V} + \|B\|_{\BMOVUT}^p \bigg)^\frac{1}{p}
 \\ & \approx \Ap{W} ^\f{1}{p}
\\ &  \lesssim \max_{1 \leq s \leq N}  \|R_s \otimes \mathbf{I}_{2m}\|_{L^p(W) \rightarrow L^p(W)}
\\ & \leq \max_{1 \leq s  \leq N} \pr{\|[M_B, R_s \otimes \mathbf{I}_m]\|_{L^p(U) \rightarrow L^p(V)} + \|R_s \otimes \mathbf{I}_m\|_{L^p(U) \rightarrow L^p(U)} + \|R_s \otimes\mathbf{I}_m\|_{L^p(V) \rightarrow L^p(V)}}
\end{align*}

Rescalling, and in particular letting $B \mapsto r B$ for $r > 0$ gives

\begin{align*}  \bigg(\Ap{U}  &  + \Ap{V} + r^p \|B\|_{\BMOVUT}^p \bigg)^\frac{1}{p}
\\ & \lesssim \max_{1 \leq s  \leq N} \pr{r \|[M_B, R_s \otimes \mathbf{I}_m]\|_{L^p(U) \rightarrow L^p(V)} + \|T\otimes\mathbf{I}_m\|_{L^p(U) \rightarrow L^p(U)} + \|T\otimes\mathbf{I}_m\|_{L^p(V) \rightarrow L^p(V)}} \end{align*}

Finally dividing both sides by $r $ and letting $r \rightarrow \infty$ gives us that

\begin{equation*} \|B\|_{\BMOVUT} \lesssim \max_{1 \leq s  \leq N} \|[M_B, R_s \otimes \mathbf{I}_m]\|_{L^p(U) \rightarrow L^p(V)} \end{equation*}

\subsection{Proof of Proposition \ref{BloomQuant}} We finally give a very short proof of Proposition \ref{BloomQuant} by estimating $\|b\|_{\BMOVUT}$. Namely,  fix a cube $Q$. Then

\begin{align*} \fint_Q & \left( \fint_Q \|U ^\frac{1}{p} (x)  U ^{-\frac{1}{p}}  (y)\|^{p'} |b(x) - b(y)|  ^{p'} \, dy \right)^\frac{p}{p'} \, dx
\\ & \lesssim  \fint_Q  \left( \fint_Q \|U ^\frac{1}{p} (x) U ^{-\frac{1}{p}}  (y)\| |b(x) - m_Q b|   ^{p'} \, dy \right)^\frac{p}{p'} \, dx
\\ & \qquad + \fint_Q  \left( \fint_Q \|U ^\frac{1}{p} (x) U ^{-\frac{1}{p}}  (y) \| |b(y) - m_Q b|  ^{p'} \, dy \right)^\frac{p}{p'} \, dx
\\ & = (A) + (B) \end{align*}

We only estimate $(A)$ as $(B)$ can be similarly estimated.  By the classical scalar sharp reverse H\"{o}lder inequality, we can pick $\epsilon  \approx \Ainf{U}{p} ^{-1}$ where for any $\v{e} \in \C$ we have \begin{equation*} \pr{\fint_Q \abs{U ^\frac{1}{p} (x) \U_Q ' \v{e}}^{\frac{p}{1-\epsilon}}} ^{1-\epsilon} \lesssim \fint_Q \abs{U ^\frac{1}{p} (x) \U_Q ' \v{e}}^p \approx \Ap{U} \end{equation*}

and therefore

\begin{align*} (A) & \leq \left( \fint_Q \|U ^\frac{1}{p} (x) \U_Q '\|^{p'}   \| (\U_Q ')^{-1}  U ^{-\frac{1}{p}}  (y) \| ^{p'} |b(x) - m_Q b|^{p'}   \, dy \right)^\frac{p}{p'} \, dx
\\& \lesssim \fint_Q \|U ^\frac{1}{p} (x) \U_Q '   \|^p |b(x) - m_Q b|^{p} \, dx
\\ & \lesssim \pr{\fint_Q \|U ^\frac{1}{p} (x) \U_Q '   \|^\frac{p}{1-\epsilon} \, dx }^{1-\epsilon} \pr{\fint_Q |b(x) - m_Q b|^{\frac{p}{\epsilon}} \, dx} ^\epsilon
\\ & \lesssim \Ap{U} \Ainf{U}{p} ^{p} \|b\|_{\text{BMO}}^{p} \end{align*} by the classical John-Nirenberg inequality.

Similarly we can estimate \begin{equation*} (B) \lesssim \Ap{U} \Ainf{U^{-\frac{p'}{p}}}{p'} ^{p} \|b\|_{\text{BMO}}^{p} \end{equation*}

\noindent so by our assumption on $T$ we have \begin{equation*} \| [M_b, T \otimes \mathbf{I}_{m}]  \|_{L^p (U) \rightarrow L^p(U) }  \leq \phi\pr{C\Ap{U} + C\Ap{U} \pr{  \Ainf{U^{-\frac{p'}{p}}}{p'} ^{p} + \Ainf{U}{p} ^{p} }    \|b\|_{\text{BMO}}^{p}}. \end{equation*}
Rescaling, setting $b \mapsto b \br{\pr{\Ainf{U^{-\frac{p'}{p}}}{p'}  + \Ainf{U}{p}}\|b\|_{\text{BMO}}} ^{-1}$ gives \begin{equation*} \| [M_b, T \otimes \mathbf{I}_m]   \|_{L^p (U) \rightarrow L^p(U) }  \leq \phi\pr{C\Ap{U} }   \|b\|_{\text{BMO}}  \pr{\Ainf{U^{-\frac{p'}{p}}}{p'}  + \Ainf{U}{p} }.  \end{equation*}.

%We now prove \eqref{TwoWeightQuant} using \eqref{OneWeightMixedQuant}. In particular, define $\Phi$ as in \eqref{Phi} so that \begin{equation*} \Phi %([M_b, T]_{k-1} \otimes \mathbf{I}_{2m})\Phi^{-1} = \left(\begin{array}{cc} U^\frac{1}{p} ([M_b, T]_{k-1} \otimes \mathbf{I}_{m})  U^{-\frac{1}{p}} & %V^\frac{1}{p}  ([M_b, T]_{k} \otimes \mathbf{I}_{m}) U^{-\frac{1}{p}}   \\ 0 & V^{\frac{1}{p}} ([M_b, T]_{k-1} \otimes \mathbf{I}_{2m}) V^{-\frac{1}{p}} %\end{array} \right). \end{equation*} Suppose that $b \in \text{BMO}$ and again use the polar decomposition $\Phi = \MC{U}W^\frac{1}{p}$.  But since
% $\Ainf{W}{p} \lesssim \Ap{W}$ and $\Ainf{W^{-\frac{p'}{p}}}{p'} \lesssim \Ap{W}^\frac{1}{p-1}$ we have \begin{align*} \| & [M_b, T]_{k}  \otimes %\mathbf{I}_{m}\|_{L^p(U) \rightarrow L^p(V)}
%\\ & \lesssim \|[M_b, T]_{k-1} \otimes \mathbf{I}_{2m}\|_{L^p(W) \rightarrow L^p(W)}
%\\ & \lesssim \|b\|_{\text{BMO}} ^{k-1}   \Ap{W}^{(k-1) \max\{1, \frac{1}{p-1}\}} \phi\pr{C\Ap{W} }
%\\ &  \lesssim \|b\|_{\text{BMO}} ^{k-1}      \pr{\Ap{U} +  \Ap{V} + \|b\|_{\BMOVU}}  ^{(k-1) \max\{1, \frac{1}{p-1}\}} \phi\pr{C\Ap{U} + C \Ap{V} + %C\|b\|_{\BMOVU}}. \end{align*} Rescaling $b \mapsto b \|b\|_{\BMOVU}^{-1} $ now completes the proof.

\section{Lower bound for Riesz transforms } \label{WienerSection}

In this section we prove Theorem \ref{RieszThm}.  Clearly it is enough to prove Theorem \ref{RieszThm} where $\|B\|_{\BMOVUT}$ is redefined by taking the supremum over balls instead of cubes, which will be more convenient for us.    Recall that the Riesz transform $R_\ell$ is the Calder\'{o}n -Zygmund operator with associated kernel \begin{equation*} K_\ell(x, y) = \frac{x_\ell - y_\ell}{|x-y|^{d + 1} } \end{equation*}  in the usual sense.

%As in \cite{NT}, if $W$ is an $m \times m$ matrix weight, then we say that $R_j$ is bounded on $L^p(W)$ if for any $\v{f} \in L^2 \cap L^p(W)$ and $\v{g} \in L^2 \cap L^{p'}(W^{-\frac{p'}{p}})$ we have that
%\begin{equation*} |\ip{(R_j \times \mathbf{I}_m) \v{f}}{\v{g}
%}_{L^2}| \leq \|R_j \times \mathbf{I}_m\|_{L^p(W) \rightarrow L^p(W)} \|\v{f}\|_{L^p(W)} \|\v{g}\|_{L^{p'}(W^{1 - p'})}. \end{equation*} so that in

Let $W$ be an $m \times m$ matrix weight, and assume for the moment that $R_\ell \otimes \mathbf{I}_m$ is bounded on $L^p(W)$, so that in particular if $\v{f} \in L^2 \cap L^p(W)$ and $\v{g} \in L^2 \cap L^{p'}(W^{-\frac{p'}{p}})$ both have compact support with $\text{dist} (\text{supp} \v{f}, \text{supp} \v{g}) > 0$, and if $E$ is a measurable subset of $\Rd$, then \begin{align} \label{restbdd} \left|\inrd \right. & \left. \inrd \unit_{E\times E} (x ,y) K_\ell(x, y) \ip{\v{f}(y)}{\v{g}(x)}_{\Cn} \, dy \ dx \right|
\\ & \leq \|\unit_E (R_\ell \otimes \mathbf{I}_m) \unit_E\|_{L^p(W) \rightarrow L^p(W)} \|\v{f}\|_{L^p(W)} \|\v{g}\|_{L^{p'}(W^{1 - p'})} \end{align}

 As was stated in the introduction, we will need the Schur multiplier/Wiener algebra arguments in \cite{LT}, which we quickly discuss now. In particular, we trivially have that the kernel $ e^{-i a \cdot x} K_\ell (x, y) e^{i a \cdot y}$ satisfies \eqref{restbdd} for any $a \in \Rd$. Thus, if $\rho \in L^1(\Rd)$ then Fubini's theorem says that the kernel \begin{equation*} \hat{\rho} (x - y) K_\ell (x, y) = \inrd \rho(a) K_\ell(x, y) e^{- i a \cdot (x-y)}\, da  \end{equation*} satisfies \eqref{restbdd} with $\|\unit_E (R_\ell \otimes \mathbf{I}_n) \unit_E\|_{L^p(W) \rightarrow L^p(W)}$ replaced with $\|\unit_E (R_\ell \otimes \mathbf{I}_n) \unit_E\|_{L^p(W) \rightarrow L^p(W)} \|\rho\|_{L^1(\Rd)}$ (where here $\hat{\rho}(s) = \inrd \rho(a) e^{-i s \cdot a} \, da$.)

Let $W_0(\Rd)$ denote the Wiener algebra defined by $W_0 (\Rd) = \{\psi = \hat{\rho} : \rho \in L^1(\Rd)\}$.  Then since $\hat{\rho} (\cdot /\varepsilon) = \widehat{\varepsilon^d \rho (\varepsilon \cdot)}$ and \begin{equation*} \inrd |\varepsilon ^d \rho(\epsilon x)| \, dx = \|\rho\|_{L^1(\Rd)} \end{equation*} we have the following result which is similar to Lemma $2.1$ in \cite{LT}.

\begin{lm} \label{LTLem} If $\psi \in W_0 (\mathbb{Rd}), \varepsilon > 0$, and $E$ is a measurable subset of $\Rd$ then \begin{align} \label{restSchurbdd} \left|\inrd \right. & \left.\inrd \psi \left(\frac{x - y}{\varepsilon} \right) \unit_{E \times E} (x,y) K_\ell(x, y) \ip{\vec{f}(y)}{\vec{g}(x)}_{\Cn} \, dy \ dx \right| \\ & \leq \|\unit_E (R_\ell \otimes \mathbf{I}_m) \unit_E\|_{L^p(W) \rightarrow L^p(W)}  \|\v{f}\|_{L^p(W)} \|\v{g}\|_{L^{p'}(W^{1 - p'})} \end{align} for any $\v{f} \in L^2 \cap L^p(W)$ and $\v{g} \in L^2 \cap L^{p'}(W^{-\frac{p'}{p}})$  of compact support with $\text{dist} (\text{supp} \ \v{f}, \text{supp} \ \v{g}) > 0$. \end{lm}

We will need three more lemmas to show that the Riesz transforms satisfy \eqref{LBO}, the first of which is probably known (though we provide a proof for the sake of completion) and the second being from \cite{LT}.

\begin{lm} \label{SchurMultLem} If $\phi \in C_c ^\infty (\Rd)$ then $|x| \phi(x) \in W_0(\Rd)$. \end{lm}

\begin{proof} The proof is similar the short proof of Lemma $2$ in \cite{DT}. Let $F(x) = |x| \phi(x)$ and pick $1 < \delta < \min\{1 + \frac{1}{d-1}, 2\}$. If $\alpha \in \{0, 1\}^d$ and \begin{equation*} P_\alpha =  \{x \in \Rd : (-1)^{\alpha_j} |x_{\alpha_j}| \leq (-1)^{\alpha_j}\} \end{equation*} then \begin{align*} \|\hat{F}\|_{L^1(\Rd)} & = \sum_{\alpha \in \{0, 1\}^d} \int_{P_\alpha} |x^\alpha|^{-1} \left( |x^\alpha| |\hat{F} (x)| \right) \, dx \\ & \leq \sum_{\alpha \in \{0, 1\}^d} \left(\int_{P_\alpha} |x^\alpha|^{-\delta} \, dx \right)^\frac{1}{\delta} \left(\inrd |x^\alpha \hat{F} (x)|^{\delta'} \, dx \right)^\frac{1}{\delta'} \\ & \lesssim \left(\inrd |\widehat{ D^\alpha F} (x)|^{\delta'} \, dx \right)^\frac{1}{\delta'} \\ & \lesssim
\left(\inrd |{ D^\alpha F} (x)|^{\delta} \, dx \right)^\frac{1}{\delta} \end{align*} where in the last inequality we used the classical Hausdorff-Young inequality.  However,  an elementary estimate using the Leibniz formula tells us that \begin{equation*} |D^\alpha F(x)| \lesssim |x|^{1 - |\alpha|}.   \end{equation*} Thus, the fact that $1 < \delta < 1 + \frac{1}{d-1}$ gives us that $\|\hat{F}\|_{L^1(\Rd)} < \infty$, which by Fourier inversion completes the proof. \end{proof}

\begin{lm}[Lemma $3.1$, \cite{LT}] There exists Borel sets $E_k ^1$ and $E_k ^2$ such that \begin{list}{}{}
\item $(i) \ $ For all $k \in \N$ we have $\text{dist} (E_k ^1, E_k ^2) > 0$.
\item $(ii) \ $ The operators defined by $P_k ^\ell \v{f} = \unit_{E_k ^\ell} \v{f}$ for $\ell = 1, 2$ converge to $\frac12 \text{Id}$ in the $L^2(\Rd)$ weak operator topology.
\item $(iii) \ $ For any $1 \leq p < \infty$ and $\ell = 1, 2$ we have \begin{equation*} \lim_{k \rightarrow \infty}   \|\unit_{E_{k}^\ell} \v{f} \|_{L^p(\Rd)} = 2^{-\frac{1}{p}} \|\v{f}\|_{L^p(\Rd)} \end{equation*}
    \end{list} \end{lm}

We now need to introduce the concept of a reducing matrix.  Namely, for a set $Q$ of finite nonzero measure,  let $\U_Q, \V_Q, \U_Q ', \V_Q ' $ respectively be positive definite matrices where for any $\vec{e} \in \Cn$ we have \begin{equation} | \U_Q \vec{e} |^p \approx \fint_Q | U^\frac{1}{p} (x) \vec{e} |^p \, dx, \ \ \  | \U_Q '  \vec{e} |^{p'} \approx \fint_Q |U^{-\frac{1}{p}} (x) \vec{e} |^{p'} \, dx \label{ReducingOp} \end{equation} and a similar statement holds for $\V_Q$ and $\V_Q '$ with respect to $V$.  Despite its perhaps abstract appearance, the reader should think of $\U_Q$ as ``the $L^p$ average of $U^\frac{1}{p}$ over $Q$" and should similarly  of $\U_Q'$ as an average. In fact, if $U$ is a matrix A${}_p$ weight then it can be shown (see Lemma $2.2$ in \cite{IKP}) that for any $\v{e} \in \Cn$ \begin{equation} \label{RedOpAveEq} |\U_Q  \v{e}| \approx | m_Q (U^{\frac{1}{p}}) \v{e}| \end{equation}  and a similar result holds for $\U_Q'$ since $U$ is a matrix A${}_p$ weight if and only if $W^{-\frac{p'}{p}}$ is a matrix A${}_{p'}$ weight.  Also note that we can in fact rewrite the matrix A${}_p$ condition as \begin{equation*} \Ap{U} \approx \sup_{\substack{ Q \subseteq \Rd \\ Q \text{ is a cube }} } \|\U_Q \U_Q '\|^p \end{equation*} and thus since $\|\U_Q \U_Q '\| = \|\U_Q '  \U_Q \|$, we can also write the matrix A${}_p$ condition out as \begin{equation} \label{MatrixApDualDef}\Ap{W} ^\frac{p'}{p} =  \sup_{\substack{Q \subset \R^d \\ Q \text{ is a cube}}} \fint_Q \left( \fint_Q \|W^{\frac{1}{p}} (x) W^{- \frac{1}{p}} (y) \|^{p} \, dx \right)^\frac{p'}{p} \, dy  < \infty. \end{equation}

 Furthermore, note that we can rewrite the $\BMOVU$ condition much more naturally as

\begin{equation*} \|B\|_{\BMOVU} = \sup_{\substack{Q \subseteq \Rd \\ Q \text{is a cube}} }  \pr{\fint_Q \|\V_Q (B(x) - m_Q B)\U_Q ^{-1}  \| \, dx}^\f{1}{p}  < \infty \end{equation*}

  The next proposition is implicitly proved in \cite{G} (among other papers) but is not explicitly stated in the literature.

   \begin{prop} \label{AveProp} If $E \subseteq \Rd$ is a set of finite positive measure then for any matrix weight $W$ we have  \begin{equation*} \|A_E\|_{L^p(W) \rightarrow L^p(W)} \approx \norm{\MC{W}_E ' \MC{W}_E } \end{equation*} where $A_E$ is the averaging operator \begin{equation*} A_E \v{f} = \unit_E \fint_E \v{f} (x) \, dx \end{equation*} \noindent and $\MC{W}_E$ and $\MC{W}_E '$ are reducing matrices for $W$. \end{prop}

   \begin{proof}  Let $\rho_{p, E} $ be the norm on $\Cn$ defined by \begin{equation*} \rho_{p, E} (\v{e}) = \pr{\fint_E \abs{W^\frac{1}{p} (x) \v{e}}^p \, dx }^\frac{1}{p} \approx |\MC{W}_E \v{e}| . \end{equation*} A standard duality argument immediately says that \begin{equation*} \rho_{p, E} (\v{e}) = \sup_{\v{v} \neq 0} \frac{\ip{\v{e}}{\v{v}}_{\Cn}}{\rho_{p, E} ^* (\v{v})} \end{equation*} where \begin{equation*} \rho_{p, E} ^* (\v{e}) = \sup_{\v{v} \neq 0} \frac{\ip{\v{e}}{\v{v}}_{\Cn}}{\rho_{p, E}  (\v{v})} \approx \abs{\MC{W}_E ^{-1} \v{e}}. \end{equation*}   Using these facts in conjunction with the fact that $(L^p(W))^* = L^{p'} (W^{-\frac{p'}{p}})$ under the unweighted $L^2$ inner product, we get that

   \begin{align*}  \sup_{\|\v{f}\|_{L^p(W)} = 1} \norm{\unit_E \fint_E \v{f}(x) \, dx}_{L^p(W)}
   & = |E|^{-\frac{1}{p'}} \sup_{\|\v{f}\|_{L^p(W)} = 1}  \sup_{\v{e} \neq 0}  \frac{\int_E \ip{\v{f}(x)}{\v{e}}_{\Cn}}{\rho_{p, E} ^* (\v{v})}
   \\ & = |E|^{-\frac{1}{p'}}\sup_{\v{e} \neq 0} \frac{\norm{\unit_E \v{e}}_{L^{p'}(W^{-\frac{p'}{p}})}}{\rho_{p, E} ^* (\v{v})}
   \approx \sup_{\v{e} \neq 0} \frac{\norm{\MC{W}_E '\v{e}}}{\norm{\MC{W}_E ^{-1} \v{e}}}. \end{align*}  Replacing $\v{e}$ by $\MC{W}_E \v{e} $ completes the proof.
   \end{proof}

   Putting together everything in this section gives us the following crucial Lemma

   \begin{lm} \label{AveLem} Let $\mathscr{B}$ be a ball and  $E \subseteq \mathscr{B}$ have nonzero finite measure.  Then   \begin{equation*} \|\MC{W}_{E}' \MC{W}_{E}\| \lesssim\frac{|\mathscr{B}|}{|{E}|} \max_{1 \leq \ell \leq d}  \|\unit_{E}(R_\ell \otimes \mathbf{I}_m) \unit_{E} \|_{L^p(W) \rightarrow L^p(W)}\end{equation*}
\end{lm}

\begin{proof}

Let $\MC{B}$ be a ball with radius $\varepsilon > 0$. We will only consider the case that $d$ is even, since the case that $d$ is odd is much easier and does not require Lemma \ref{SchurMultLem}.    Let $\phi \in C_c^\infty(\Rd)$ satisfy $\phi(x) = 1$ if $|x| < 2$,  so by Lemma \ref{SchurMultLem} and the fact that $x_\ell |x|^{d - 2} \phi(x) \in C_c ^\infty(\Rd) \subseteq W_0(\Rd)$ we get that $ x_\ell   |x|^{d-1} \phi ^2 (x) \in W_0 (\Rd)$.  Using Lemma \ref{LTLem} and summing over $\ell$ then gives us that the kernel \begin{equation*} \varepsilon ^{- d } \phi ^2 \left(\frac{x - y}{\varepsilon} \right) \unit_{\mathscr{B} \times \mathscr{B}} (x, y)  = \frac{c_d}{|\mathscr{B}|}  \unit_{\mathscr{B} \times \mathscr{B}} (x, y) \end{equation*}  replacing $\psi\left(\frac{x-y}{\epsilon}\right) \unit_{E \times E} (x, y) K_{\ell}(x, y)$ satisfies \eqref{restSchurbdd}.
  Thus, if $P_k ^\ell $ are the projections from the previous lemma then for any $\v{f} \in L^2 \cap L^p(W)$ and $ \v{g} \in L^2 \cap L^{p'} (W^{1-p'})$ we have \begin{align*} \abs{\ip{A_{E} \v{f}}{\v{g}}_{L^2}} & = \frac{|\mathscr{B}|}{|{E}|} \lim_{k \rightarrow \infty}  \abs{\ip{\unit_{E} A_\mathscr{B} \unit_{E} P_k ^1  \v{f}}{P_k ^2  \v{g}}_{L^2}} \\ & \lesssim
    \frac{|\mathscr{B}|}{|{E}|} \max_{1 \leq \ell \leq d}  \|\unit_{E} (R_\ell \otimes \mathbf{I}_m) \unit_{E} \|_{L^p(W) \rightarrow L^p(W)}   \lim_{k \rightarrow \infty} \|P_k ^1 \v{f}\|_{L^p(W)} \| P_k ^2 \v{g} \|_{L^{p'} (W^{1-p'})} \\ & = \frac{|\mathscr{B}|}{|{E}|} \max_{1 \leq \ell \leq d}  \|\unit_{E} (R_\ell \otimes \mathbf{I}_m) \unit_{E}\|_{L^p(W) \rightarrow L^p(W)} \| \v{f}\|_{L^p(W)} \| \v{g} \|_{L^{p'} (W^{1-p'})}. \end{align*}  However, since bounded functions with compact support are dense in $L^p(W)$ and $ L^{p'} (W^{1-p'})$,  Proposition \ref{AveProp} then says that

    \begin{equation*} \|\MC{W}_{E} '  \MC{W}_{E}\| \approx \norm{A_E}_{L^p(W) \rightarrow L^p(W)} \lesssim \frac{|\mathscr{B}|}{|{E}|} \max_{1 \leq \ell \leq d}  \|\unit_{E}(R_\ell \otimes \mathbf{I}_m) \unit_{E} \|_{L^p(W) \rightarrow L^p(W)}\end{equation*}

\end{proof}

We now finish the proof of Theorem \ref{RieszThm}.  Fix a ball $\mathscr{B}$ and define $E_M = \{x \in \mathscr{B} : \max\{\norm{U(x)}, \norm{U^{-1}(x)}, \norm{V(x)}, \norm{V^{-1}(x)}\} < M\}$ where $M > 0$ is large enough that $2 |E_M| > |\mathscr{B}|$.

Defining \begin{equation*} \AP{W}{E} =  \fint_{E} \left( \fint_{E} \|W^{\frac{1}{p}} (x) W^{- \frac{1}{p}} (y) \|^{p'} \, dy \right)^\frac{p}{p'} \, dx  \end{equation*} and \begin{equation*} \|B\|_{\MOVUT{{E}}} ^p  = \fint_{E} \pr{\fint_{E} \norm{V^\frac{1}{p} (x) (B(x) - B(y)) U^{-\frac{1}{p}}(x) }^{p'} \, dy }^\frac{p}{p'} \, dx \end{equation*} and also defining $W$ and $\Phi$ as in the beginning of Section \ref{UBSection}, we have using Lemma \ref{AveLem} with respect to $E = E_M$ that \begin{align*} \bigg(\AP{U}{E_M} & + \AP{V}{E_M} + \|B\|_{\MOVUT{E_M}}^p \bigg)^\frac{1}{p}
 \\ & \approx \AP{W}{E_M} ^\f{1}{p}
\\ & \approx \|\MC{W}_{E_M} ' \MC{W}_{E_M}\|
\\ & = \|\MC{W}_{E_M}  \MC{W}_{E_M} ' \|
\\ & \lesssim \max_{1 \leq \ell \leq d}  \|\unit_{E_M} (R_\ell \otimes \mathbf{I}_{2m}) \unit_{E_M} \|_{L^p(W) \rightarrow L^p(W)}
\\ & \lesssim \max_{1 \leq \ell  \leq d} \left(\| [M_B, R_\ell\otimes \mathbf{I}_{2m}]\|_{L^p(U) \rightarrow L^p(V)} \right. +
\\ & \qquad + \left.\|\unit_{E_M}(R_\ell \otimes \mathbf{I}_{2m})\unit_{E_M}\|_{L^p(U) \rightarrow L^p(U)} + \|\unit_{E_M}(R_\ell \otimes \mathbf{I}_{2m})\unit_{E_M}\|_{L^p(V) \rightarrow L^p(V)}.\right) \end{align*}

Notice that all quantities above are bounded as all weights involved are pointwise bounded in norm and we assume $\| [M_B, R_\ell\otimes \mathbf{I}_{2m}]\|_{L^p(U) \rightarrow L^p(V)} < \infty$.  Thus, as was done in the proof of Theorem \ref{IntLB}, we can rescale and set $B \mapsto rB$, divide by $r$, and let $r \rightarrow \infty$ to get that

\begin{equation*} \|B\|_{\MOVUT{E_M}} \lesssim \max_{1 \leq \ell  \leq d} \| [M_B, R_\ell\otimes \mathbf{I}_m]\|_{L^p(U) \rightarrow L^p(V)}. \end{equation*} First letting $M \rightarrow \infty$ and using Fatou's lemma, and then taking the supremum over all balls $\mathscr{B}$ shows that $$ \|B\|_{\BMOVUT} \lesssim \max_{1 \leq \ell \leq d} \|[M_B, R_\ell]\|_{L^p(U) \rightarrow L^p(V)} $$ as desired.    To show the same estimate is true with $\|B\|_{\BMOVUTd}$  expand out $\AP{W}{E_M} ^\f{1}{p}
 \approx \|\MC{W}_{E_M} '  \MC{W}_{E_M}\| $ using the reducing matrix $\MC{W}_{E_M} $ first and repeat the arguments above, which completes the proof of \eqref{RieszThmA}.

\section{John Nirenberg theorems} \label{JNSection}

 We will finish this paper by proving the equivalency between $\|B\|_{\BMOVUT}$ and $\|B\|_{\BMOVU}$ when $U$ and $V$ are matrix A${}_p$ weights.  Note that we will not track the $\Ap{U}$ and $\Ap{V}$ dependence of our constants because we will need to use the lower matrix weighted Triebel-Lizorkin bounds from \cite{NT,V} when $d = 1$ and $d > 1$ in \cite{I2}, which are most likely far from sharp.     We will need the following simple result that is a special case of Theorem $2.2$ in \cite{I} and proved using a simple idea from \cite{IKP}.    Note that throughout this section $\D$ will refer to some dyadic lattice of cubes in $\Rd$.

%\begin{lm}  \label{SharpMaxLM} Let $M_U ' \v{f}$  be the Goldberg intermediary maximal function defined by \begin{equation*} M_U ' \v{f} (x) = \sup_{\D %\ni Q \ni x}  m_Q |\U_Q U^{-\frac{1}{p}} \vec{f}| \end{equation*}  It is easy to see (and well known, see \cite{G}) that \begin{equation*} \|M_U '\|_{L^p %\rightarrow L^p} \lesssim \Ap{U} ^{\frac{1}{p}}  \Apinf{U}{p}^\frac{1}{p} \end{equation*}
%  \end{lm}

\

\begin{prop} \label{embedthm} Let $U$ be a matrix A${}_p$ weight and let $A = \{a_Q\}_{Q \in \MC{D}} $ be a nonnegative Carleson sequence of scalars, meaning that \begin{equation*} \|A\|_* ^2 =\sup_{J \in \D} \frac{1}{|J|} \sum_{Q \in \D(J)} a_Q ^2 < \infty \end{equation*} Then for any $\vec{f} \in L^p$ we have \begin{equation} \label{embedineq}\pr{\inrd  \left(\sum_{Q \in \D} \frac{   [ a_Q m_Q |\U_Q U^{-\frac{1}{p}} \vec{f}|]^2}{|Q|} \unit_Q (x) \right)^\frac{p}{2} \, dx}^\frac{1}{p} \lesssim
 \|A\|_*  \|\vec{f}\|_{L^p}  \end{equation} \end{prop}

\begin{proof}

 Let \begin{equation*} \tilde{A} =  \sum_{\varepsilon \in \S} \sum_{Q \in \D} a_Q h_Q ^\varepsilon \end{equation*}  where  $\S = \{1, 2, \ldots, 2^{d} - 1\}$ and $\{h_Q ^\varepsilon \}_{\{Q \in \D, \varepsilon \in \S\}}$ is any Haar system on $\Rd$.   Clearly for any $\D \ni Q \ni x$ we have that \begin{equation*} m_Q |\U_Q U^{-\frac{1}{p}} \vec{f}| \leq m_Q M_U ' \vec{f} \end{equation*}  where $M_U ' \v{f}$  is the ``Goldberg intermediary maximal function" defined by \begin{equation*} M_U ' \v{f} (x) = \sup_{\D \ni Q \ni x}  m_Q |\U_Q U^{-\frac{1}{p}} \vec{f}| \end{equation*}

 Thus, since $M_U ' : L^p \rightarrow L^p$ for matrix A${}_p$ weights $U$ (see \cite{G}, p. 8), we have that that \begin{align*}  \int_{\Rd}  & \left( \sum_{\varepsilon \in \S} \sum_{Q \in \D} \frac{  [a_Q m_Q M_U ' \vec{f}]^2 }{|Q|} \unit_Q(x) \right)^\frac{p}{2} \, dx
 \\ & \lesssim \|\pi_{\tilde{A}} M_U ' \vec{f}\|_{L^p} ^p   \lesssim \|A\|_* ^p \|M_U ' \vec{f}\|_{L^p} ^p  \lesssim   \|\v{f}\|_{L^p} ^p \end{align*}
 by unweighted dyadic Littlewood-Paley theory, where here $\pi_{\tilde{A}}$ is the paraproduct \begin{equation*} \pi_{\tilde{A}} g(x) = \sum_{\varepsilon \in \S}  \sum_{Q \in \D} m_Q g \tilde{A} _Q  ^\varepsilon   h_Q ^\varepsilon (x) \end{equation*}

 \end{proof}

The following is the key to proving the equivalence between $\|B\|_{\BMOVUT}$ and $\|B\|_{\BMOVU}$. Note that the lemma below was implicitly proved in \cite{I2} though not explicitly stated, and therefore for the sake of completion we will include the details.

\begin{lm}\label{eJNLem} If $U$ and $V$ are matrix A${}_p$ weights then there exists $\epsilon > 0$ small enough where for any $0 < \epsilon' < \epsilon$ we have \begin{equation*} \sup_{I \in \D(J)} \left( \fint_I \| V^\frac{1}{p} (x) (B(x) - m_I B ) \U_I ^{-1} \| ^p \, dx \right)^\frac{1}{p} \leq C \sup_{I \in \D(J)} \left(\fint_I  \| \V_I (B(x) - m_I B) \U_I  ^{-1}  \|^{1 + \epsilon'}  \, dx \right)^\frac{1}{1 + \epsilon'} \end{equation*}  and \begin{equation*}\sup_{I \in \D(J)} \left(\fint_I \|\V_I (B(x) - m_I B) \U_I ^{-1}\| ^{1 + \epsilon} \, dx \right)^\frac{1}{1+ \epsilon} \leq C \sup_{I \in \D(J)} \, \fint_I \| V^\frac{1}{p} (x) (B(x) - m_I B ) \U_I ^{-1} \| ^p \, dx. \end{equation*} where $C$ is independent of $B$ and $J$ (but depends on $\epsilon'$.) \end{lm}

We prove Lemma \ref{eJNLem} through a series of lemmas.

\begin{lm} \label{JN1} If $U, V$ are matrix A${}_p$ weights and $B$ is locally integrable then \begin{equation*} \sup_{I \in \D(J)} \fint_I \| V^\frac{1}{p} (x) (B(x) - m_I B) \U_I ^{-1}\| ^p \, dx \leq C \sup_{I \in \D(J)} \left( \frac{1}{|I|} \sum_{\substack{ Q \in \D(I)  \\ \varepsilon \in \S}} \| \V_Q B_Q ^\varepsilon \U_Q ^{-1} \| ^2\right)^\frac12 \end{equation*}  where $C$ is independent of $B$ and $J$.
\end{lm}

\begin{proof} Let $I \in \D(J)$.  By the Triebel-Lizorkin embedding (see \cite{NT, V} for $d = 1$ and \cite{I2} for $d > 1$) we have that \begin{align*} \fint_I &  \| V^\frac{1}{p} (x) (B(x) - m_I B) \U_I ^{-1}\| ^p \, dx  \\ & \lesssim  \fint_I \left(\sum_{\substack{ Q \in \D(I)  \\ \varepsilon \in \S}}  \frac{\|\V_Q B_Q \U_I ^{-1}\|^2}{|Q|} \unit_Q (t) \right)^\frac{p}{2} \, dt
\\ & \leq \fint_I \left(\sum_{\substack{ Q \in \D(I)  \\ \varepsilon \in \S}}  \frac{\|\V_Q B_Q \U_Q^{-1} \|^2} {|Q|} \|m_Q [(\U_Q U^{-\frac{1}{p}}) U^\frac{1}{p} \U_I ^{-1} \unit_I]\|^2 \unit_Q (t) \right)^\frac{p}{2} \, dt
\\ & \lesssim  \left(\sup_{I \in \D(J)} \frac{1}{|I|} \sum_{\substack{ Q \in \D(I)  \\ \varepsilon \in \S}} \| \V_Q B_Q ^\varepsilon \U_Q ^{-1} \| ^2\right)^\frac{p}{2} \fint_I \|U^\frac{1}{p} \U_I ^{-1}\|^p \, dx \end{align*} where in the last line we used Proposition \ref{embedthm}.  \end{proof}

\begin{lm} \label{JN2}

For $\epsilon > 0$ small enough (independent of $ B$) we have \begin{equation*} \left( \fint_I \|\V_I (B(x) - m_I B) \U_I ^{-1}\| ^{1 + \epsilon} \, dx \right)^\frac{1}{1+ \epsilon} \leq C \left( \fint_I \| V^\frac{1}{p} (x) (B(x) - m_I B ) \U_I ^{-1} \| ^p \, dx \right)^\frac{1}{p} \end{equation*} where $C$ is independent of $I$ and $B$.  \end{lm}

\begin{proof} \begin{align*} \biggl(   \fint_I & \|\V_I (B(x) - m_I B) \U_I ^{-1}\| ^{1 + \epsilon} \, dx \biggr)^\frac{1}{1+ \epsilon} \\ & \leq \left(\fint_I \|\V_I V^{-
 \frac{1}{p}} (x) \|^{1 + \epsilon} \| V^\frac{1}{p} (x) (B(x) - m_I B) \U_I ^{-1}\| ^{1 + \epsilon} \, dx \right)^\frac{1}{1+ \epsilon} \\ & \leq \left(\fint_I \| \V_I V^{-\frac{1}{p}} (x) \| ^\frac{p(1 + \epsilon)}{p - 1 - \epsilon} \, dx \right) ^{\frac{1}{1  +  \epsilon} (\frac{p - 1 - \epsilon}{p})} \left(\fint_I\| V^\frac{1}{p} (x) (B(x) - m_I B ) \U_I ^{-1} \| ^p \, dx \right)^\frac{1}{p} \\ & \leq \Ap{V}^\frac{1}{p}  \left(\fint_I \| V^\frac{1}{p} (x) (B(x) - m_I B ) \U_I ^{-1} \| ^p \, dx \right)^\frac{1}{p} \end{align*} for $\epsilon > 0$ small enough by the reverse H\"{o}lder inequality.
\end{proof}

We now recall the two matrix weighted stopping time from \cite{I2} which is a modification of the one matrix weighted stopping time from \cite{IKP}.   Finally assume that $U, V$ are a matrix A${}_p$ weights and that $\lambda$ is large. For any cube $I \in \D$, let $\J(I)$ be the collection of maximal $J \in \D(I)$ such that either of the two conditions \begin{equation*}      \|\U_I  \U_J ^{-1}\|  > \lambda \    \text{   or   }  \  \|\V_I ^{-1} \V_J\|  > \lambda.  \end{equation*}   Also, let $\F(I)$ be the collection of dyadic subcubes of $I$ not contained in any cube $J \in \J(I)$, so that clearly $J \in \F(J)$ for any $J \in \D$.

Let $\J^0 (I) := \{I\}$ and inductively define $\J^j(I)$ and $\F^j(I)$ for $j \geq 1$ by \begin{equation*} \J^j (I) := \{ R \in \J(Q) : Q \in \J^{j - 1}(I)\} \end{equation*}  and  $\F^j (I) = \{J' \in \F(J) : J \in \J^{j - 1}(I)\}$. Clearly the cubes in $\J^j(I)$ for $j > 0$ are pairwise disjoint.  Furthermore, since $J \in \F(J)$ for any $J \in \D(I)$, we have that $\D(I) = \bigcup_{j = 1}^\infty \F^j(I)$ and that the collections $\F^j(I)$ are disjoint.  We will slightly abuse notation and write $\bigcup \J(I)$ for the set $\bigcup_{J \in \J(I)} J$ and write $|\bigcup \J(I)|$ for $|\bigcup_{J \in \J(I)} J|$. By easy arguments (see \cite{I2}) we can pick $\lambda$ depending on $U$ and $V$  so that \begin{equation} \label{decay} |\bigcup \J ^j (I)| \leq 2^{-j} |I| \end{equation}  for every $I \in \D$.

\begin{lm} \label{JN3} If $U, V$ are matrix $A_p$ weights and $0 < \epsilon' \leq 1$ then there exists $C > 0$ independent of $J$ and $B$ where \begin{equation*} \sup_{I \in \D(J)} \left( \frac{1}{|I|} \sum_{\substack{ Q \in \D(I) \\  \varepsilon \in \S}} \| \V_Q B_Q ^\varepsilon \U_Q ^{-1} \| ^2\right)^\frac12 \leq C  \sup_{I \in \D(J)} \left(\fint_I \| \V_I (B(x) - m_I B) \U_I  ^{-1}  \|^{1 + \epsilon'}  \, dx \right)^\frac{1}{1 + \epsilon'}. \end{equation*}\end{lm}

\begin{proof}
Fix $I \in \D(J)$.  By the classical unweighted John-Nirenberg theorem and by unweighted dyadic Littlewood-Paley theory, it is enough to prove that \begin{equation*} \fint_I \left( \sum_{\substack{ Q \in \D(I) \\ \varepsilon \in \S}} \frac{\| \V_Q B_Q ^\varepsilon \U_Q ^{-1} \| ^2}{|Q|} \unit_Q(x) \right)^\frac{1 + \epsilon'}{2} \, dx    \leq C  \sup_{I' \in \D(J)} \fint_{I'} \| \V_{I'} (B(x) - m_{I'} B) \U_{I'}  ^{-1}  \|^{1 + \epsilon'}  \, dx. \end{equation*} for $I \in \D(J)$ where $C$ is independent of $I, J$ and $B$.  To that end  \begin{align*} \fint_I & \left( \sum_{\substack{Q \in \D(I) \\ \varepsilon \in \S}} \frac{\| \V_Q B_Q ^\varepsilon \U_Q ^{-1} \| ^2}{|Q|} \unit_Q(x) \right)^\frac{1 + \epsilon'}{2} \, dx \\ & \leq \fint_I\left( \sum_{j = 1}^\infty \sum_{K \in \J^{j-1}(I)} \sum_{\substack{Q \in \F(K) \\ \varepsilon \in \S}} \frac{(\| \V_Q \V_K^{-1} \| \| \V_K B_Q ^\varepsilon \U_K ^{-1} \| \U_K \U_Q ^{-1} \|)^2}{|Q|} \unit_Q(x) \right)^\frac{1 + \epsilon'}{2} \, dx \\ &   \leq {C}\fint_I \left( \sum_{j = 1}^\infty \sum_{K \in \J^{j-1}(I)} \sum_{\substack{Q \in \F(K) \\ \varepsilon \in \S}} \frac{ \| \V_K B_Q ^\varepsilon \U_K ^{-1} \| ^2}{|Q|} \unit_Q(x) \right)^\frac{1 + \epsilon'}{2} \, dx \\ & \leq  \frac{C}{|I|} \sum_{j = 1}^\infty \sum_{K \in \J^{j-1}(I)}   \int_K \left( \sum_{\substack{Q \in \D(K) \\ \varepsilon \in \S}} \frac{ \| \V_K B_Q ^\varepsilon \U_K ^{-1} \| ^2}{|Q|} \unit_Q(x) \right)^\frac{1 + \epsilon'}{2} \, dx \\ & \leq \frac{C}{|I|} \sum_{j = 1}^\infty \sum_{K \in \J^{j-1}(I)}|K| \left( \fint_K \| \V_K (B(x) - m_K B ) \U_K ^{-1} \| ^{1  + \epsilon'} \, dx  \right) \\ & \leq C \sup_{I' \in \D(J)} \left(\fint_{I'} \| \V_{I'} (B(x) - m_{I'} B) \U_{I'}  ^{-1}  \|^{1 + \epsilon'}  \, dx \right)|I|^{-1} \sum_{j = 1}^\infty \sum_{K \in \J^{j-1}(I)}|K|
 \\ & \leq  C \sup_{I' \in \D(J)} \fint_{I'} \| \V_{I'} (B(x) - m_{I'} B) \U_{I'}  ^{-1}  \|^{1 + \epsilon'}  \, dx.   \end{align*}
\end{proof}

  The proof of Lemma \ref{eJNLem} now follows immediately by combining Lemmas \ref{JN1}, \ref{JN2}, and \ref{JN3}.    In particular, let $\epsilon > 0$ be from Lemma \ref{JN2}. Then for any $0 < \epsilon' < \epsilon$ we have \begin{align*} \sup_{I \in \D(J)} \left(\fint_I \|\V_I (B(x) - m_I B) \U_I ^{-1}\| ^{1 + \epsilon} \, dx \right)^\frac{1}{1+ \epsilon} & \leq C \sup_{I \in \D(J)} \left(\fint_I \| V^\frac{1}{p} (x) (B(x) - m_I B ) \U_I ^{-1} \| ^p \, dx \right)^\frac{1}{p}   \\ & \leq C \sup_{I \in \D(J)} \left(\frac{1}{|I|} \sum_{\substack{Q \in \D(I) \\  \varepsilon \in \S}} \| \V_Q B_Q ^\varepsilon \U_Q ^{-1} \| ^2\right)^\frac{1}{2} \\ & \leq C \sup_{I \in \D(J)} \left(\fint_I \| \V_I (B(x) - m_I B) \U_I  ^{-1}  \|^{1 + \epsilon'}  \, dx \right)^\frac{1}{1 + \epsilon'}. \end{align*} \hfill $\square$

We now prove the following,

\begin{lm} \label{TreilLem} If $U, V$ are matrix A${}_p$ weights and $\epsilon' > 0$ is small enough then \begin{equation*} \sup_{I \in \D} \left(\fint_I \| \V_I (B(x) - m_I B) \U_I  ^{-1}  \|^{1 + \epsilon'}  \, dx \right)^\frac{1}{1 + \epsilon'} \lesssim  \sup_{I \in \D} \fint_I \| \V_I (B(x) - m_I B) \U_I  ^{-1}  \|  \, dx    . \end{equation*} \end{lm}

  \begin{proof} For fixed $R \in \mathbb{N}$ let $P_R$ be the canonical projection operator \begin{equation*}  P_R B (x) = \sum_{\substack{I \in \D \\ |I| > 2^{-R}}} \sum_{\varepsilon \in \S}  B_I ^\varepsilon h_I ^\varepsilon = \sum_{\substack{I \in \D \\ |I| = 2^{-R}}} \unit_I m_I B \end{equation*}  which is trivially bounded on $L^p(\Rd)$ for $1 \leq p < \infty$.  For $I \in \D$ let \begin{equation*} F_I (x) = \unit_I (x)\V_I (B(x) - m_I B) \U_I  ^{-1}, \end{equation*} let \begin{equation*} F_I ^R (x) = \unit_I (x) \V_I (P_R B(x) - m_I (P_R B)) \U_I  ^{-1} = P_R [\unit_I \V_I (B - m_I B) \U_I ^{-1}], \end{equation*} and let $d\mu_I (x) = |I|^{-1} \unit_I (x) \, dx $.  Fix $J \in \D$ so trivially \begin{equation*} \sup_{Q \in \D(J)} \|F_Q ^R \|_{L^{1 + \epsilon'} (d\mu_Q)}   \, dx \leq \sup_{Q \in \D(J)} \|F_Q ^R \|_{L^\infty} =  C < \infty \end{equation*} where $C $ possibly depends on $J, R, B, U, $ and $V$.  Also, clearly \begin{equation*} \sup_{Q \in \D} \|F_Q ^R \|_{L^1(d\mu_Q)} \lesssim \|B\|_{\BMOVU} \end{equation*}  independent of $R > 0$ since $P_R$ is bounded on $L^1(\Rd)$ independent of $R$.

 Let $\epsilon > 0$ be from Lemma \ref{eJNLem} (applied to $P_R B$) and let $0 < \epsilon' < \epsilon$.   Let  $p_1 = 1 + \epsilon ', \ p_2 = 1 + \epsilon,  \alpha = \frac{p_2 - p_1}{p_2 - 1}, \ \beta = \frac{p_1 - \alpha}{\alpha}$, and let $C$ be the constant in Lemma \ref{eJNLem}. Then by a use of H\"{o}lder's inequality with respect to the conjugate exponents $\alpha^{-1}$ and $(1-\alpha)^{-1}$ we have

\begin{align*} \sup_{Q \in \D(J)} \|F_Q ^R \|_{L^{p_1} (d\mu_Q) } ^{p_1}  & \leq \sup_{Q \in \D(J)} \|F_Q ^R \|_{L^{1 } (d\mu_Q) }^\alpha  \sup_{Q \in \D(J)} \|F_Q ^R \|_{L^{p_2} (d\mu_Q)} ^{p_1 - \alpha} \\ & \leq \|B\|_{\BMOVU} ^\alpha   C ^{p_1 - \alpha} \sup_{Q \in \D(J)} \|F_Q ^R \|_{L^{p_1} (d\mu_Q)} ^{p_1 - \alpha} \end{align*}  which says that \begin{equation*}  \sup_{Q \in \D(J)} \|F_Q ^R \|_{L^{p_1} (d\mu_Q) } \leq \|B\|_{\BMOVU}  C^{\beta}. \end{equation*} Letting $R \rightarrow \infty$ first (using Fatou's lemma) and then taking the supremum over all $Q \in \D(J)$ and then all $J \in \D$ completes the proof.   \end{proof}

Combining everything we have the following corollary, which finishes the proof that the quantities $\|B\|_{\BMOVUT}$ and $\|B\|_{\BMOVU}$ are equivalent.

\begin{cor} \label{StrongJNCont}  If $U, V$ are matrix weights A${}_p$ weights then there exists $\epsilon > 0$ such that the following quantities are equivalent  \begin{list}{}{}
\item $a)  \   \displaystyle \sup_{\substack{Q \subseteq \Rd \\ Q \text{ is a cube}}} \fint_Q \| \V_Q (B(x) - m_Q B) \U_Q  ^{-1}  \| \, dx $ \\
\item $b) \  \displaystyle \sup_{\substack{Q \subseteq \Rd \\ Q \text{ is a cube}}} \left( \fint_Q \| V^\frac{1}{p} (x) (B(x) - m_Q B ) \U_Q ^{-1} \| ^p \, dx \right)^\frac{1}{p} $
\item $c) \  \displaystyle \sup_{\substack{Q \subseteq \Rd \\ Q \text{ is a cube}}} \left( \fint_Q \| U^{-\frac{1}{p}} (x) (B^*(x) - m_Q B^* ) (\V_Q ') ^{-1} \| ^{p'} \, dx\right)^\frac{1}{p'}.$
\item $d) \ \displaystyle \sup_{\substack{Q \subseteq \Rd \\ Q \text{ is a cube}}} \left(\fint_Q \left(\fint_Q \|V^\frac{1}{p} (x) (B(x) - B(y)) U^{-\frac{1}{p}} (y) \|^{p'}     \, dy \right)^\frac{p}{p'} \, dx\right)^\frac{1}{p}.  $
    \item $e) \ \displaystyle \sup_{\substack{Q \subseteq \Rd \\ Q \text{ is a cube}}} \left(\fint_Q \left(\fint_Q \|V^\frac{1}{p} (x) (B(x) - B(y)) U^{-\frac{1}{p}} (y) \|^{p}     \, dx \right)^\frac{p'}{p} \, dy \right)^\frac{1}{p'}.  $
    \end{list}
    \end{cor}

    \begin{proof}  If $t \in \{0, \frac13\}^{d}$  and $\D^t = \{2^{-k} ([0, 1) ^d + m  + (-1)^k t)  : k \in \Z, m \in \Z^d\}$, then given any cube $Q$, there exists $t \in \{0, \frac13\}^d$ and $Q_t\in \D^t$ such that $Q\subset Q_t$ and $\ell(Q_t)\leq 6\ell(Q)$. Thus, by standard arguments, it is enough to prove the equivalence of $a) - e)$ for any fixed dyadic grid.

    With this in mind, the equivalence between the supremums in $a)$ and $ b)$ follows immediately from Lemmas \ref{eJNLem} and \ref{TreilLem}.
 As for $c)$, since $U$ and $V$ are matrix A${}_p$ weights,  \begin{align*} \| \V_Q (B(x) - m_Q B) \U_Q  ^{-1}  \| & = \| \U_Q  ^{-1}  (B^*(x) - m_Q B^*) \V_Q  \| \\ & \approx \| \U_Q '  (B^*(x) - m_I B^*) (\V_Q ' )^{-1}\|   \end{align*} and since clearly $\U_Q '$ is an $L^{p'}$ reducing matrix for $U^{-\frac{p'}{p}}$ and a similarly $\V_Q '$ is an $L^{p'}$ reducing operator for $V^{-\frac{p'}{p}}$, we get that the supremum in $a)$ is equivalent to the supremum in $c)$ by using the equivalence to $b)$ with respect to the pair $V^{-\frac{p'}{p}}, U^{-\frac{p'}{p}}, $ and the exponent $p'$.

    Also,  \begin{align*} \fint_Q & \| V^\frac{1}{p} (x) (B(x) - m_Q B ) \U_Q ^{-1} \| ^p \, dx
   \\ & \leq  \fint_Q \left( \fint_Q  \| V^\frac{1}{p} (x) (B(x) - B(y) ) \U_Q ^{-1}\|\, dy \right)^p \, dx
    \\ & \leq \fint_Q \left( \fint_Q  \| V^\frac{1}{p} (x) (B(x) - B(y) ) U^{-\frac{1}{p}} (y) \| \| U^{\frac{1}{p}} (y)  \U_Q ^{-1}\|\, dy \right)^p \, dx
\\ &    \lesssim  \fint_Q \left( \fint_Q  \| V^\frac{1}{p} (x) (B(x) - B(y) ) U^{-\frac{1}{p}} (y) \|^{p'} \, dy \right)^\frac{p}{p'} \, dx \end{align*} which proves that $d)$ implies $b)$, and similarly $e)$ implies $c)$.  Finally, adding and subtracting $m_Q B$ in both $d)$ and $e)$, respectively, shows that $b)$ and $c)$ together implies both $d)$ and $e)$, which completes the proof. \end{proof}

\end{document}